\documentclass[reqno,b5paper]{amsart}
\usepackage{amssymb}
 \usepackage{amsthm}
\usepackage{graphicx}
 
\theoremstyle{plain}
\newtheorem{Theorem}{Theorem}[section]
\newtheorem{Example}{Example}[section]
\newtheorem{Corollary}{Corollary}[section]
\newtheorem{Lemma}{Lemma}[section]

\newtheorem{Definition}{Definition}[section]
\newtheorem{Remark}{Remark}[section]

\begin{document}

\title[On certain notions of precompactness, continuity and Lipschitz functions ]
{On certain notions of precompactness, continuity and Lipschitz functions}

\author{Pratulananda Das, Sudip Kumar Pal*, Nayan Adhikary}

\address{Department of Mathematics, Jadavpur University, Kolkata-700032, West Bengal, India, E-mail: pratulananda@yahoo.co.in, nayanadhikarysh@gmail.com}
\address{*Corresponding Author, Department of Mathematics, Diamond Harbour Women's University, Sarisa-743368, West Bengal, India, E-mail: sudipkmpal@yahoo.co.in}

\subjclass[2010]{Primary: 54D20; Secondary: 54C30, 54A25}

\begin{abstract}
The underlying theme of this article is a class of sequences in metric structures satisfying a much weaker kind of Cauchy condition, namely quasi-Cauchy sequences (introduced in \cite{bc}) that has been used to define several new concepts in recent articles  \cite{PDSPNA2, PDSPNA1}. We first consider a weaker notion of precompactness based on the idea of quasi-Cauchy sequences and establish several results including a new characterization of compactness in metric spaces. Next we consider associated idea of continuity, namely, ward continuous functions \cite{caka}, as this class of functions strictly lies between the classes of continuous and uniformly continuous functions and mainly establish certain coincidence results. Finally a new class of Lipschitz functions called ``quasi-Cauchy Lipschitz functions" is  introduced following the line of investigations in  \cite{Beer1,Beer2,Beer3,g1} and again several coincidence results are proved. The motivation behind such kind of Lipschitz functions is ascertained by the   observation that every real valued ward continuous function defined on a metric space can be uniformly approximated by real valued quasi-Cauchy Lipschitz functions.
\end{abstract}
\maketitle
\smallskip
\noindent{\bf\keywordsname{}:} {Quasi-Cauchy sequence, Ward continuity, Bourbaki quasi-precompact space, Quasi-Cauchy Lipschitz function.}

\section{Introduction}
The notion of continuity along with its several versions have always intrigued mathematicians. Not only such functions, but also the sequences which characterize various types of continuity have been of great research interest. It is well known that the continuous (Cauchy continuous) functions are characterized by convergent (Cauchy) sequences. A much weaker version of Cauchy sequences, namely, quasi-Cauchy sequences were introduced in 2010 \cite{bc} where the distance between two successive terms tends to zero and consequently one arrives at a stronger version of Cauchy continuous functions, namely, ward continuous functions \cite{caka} which preserve quasi-Cauchy sequences. Evidently Cauchy sequences are quasi-Cauchy but the converse is not generally true as the counter example is provided by the sequence of partial sums of the harmonic series. This and several such examples establish the important fact that the class of quasi-Cauchy sequences is much bigger than the class of Cauchy sequences, taking in the process more sequences under the purview. Understandably mathematical consequences are not analogous to the already existing notions based on Cauchy sequences, like the usual idea of precompactness, Cauchy-continuous functions or more recent idea of Cauchy-Lipschitz functions as these sequences don't have the easy features of Cauchy sequences (for example every subsequence of a Cauchy sequence is Cauchy whereas the analogous property fails for quasi-Cauchy sequences). This makes investigations of notions associated with quasi-Cauchy sequences much more challenging and non-trivial.\par
It is well known that precompact spaces are those spaces where every sequence has a Cauchy subsequence. Considering quasi-Cauchy sequences instead, we consider a weaker notion of precompactness, in our context which we call ``Bourbaki quasi-precompact space", though it is known as $\alpha$-bounded space in literature(see \cite{tash1,tash2,g3} for easy reference). Section 2 of this paper is devoted to the study of this notion and many of its consequences in new aspects. We define Bourbaki quasi-Cauchy sequences in line of Bourbaki-Cauchy sequences \cite{g2} and establish in Theorem \ref{a} that a sequence is Bourbaki quasi-Cauchy iff it is a subsequence of some quasi-Cauchy sequence. Also we define corresponding completeness, namely Bourbaki quasi-completeness (i.e., a space where every Bourbaki quasi-Cauchy sequence has a cluster point). Other main observations of this section are:\par
$(1)$ A space is Bourbaki quasi-precompact iff every sequence has a Bourbaki quasi-Cauchy subsequence.\par
$(2)$ A space is compact iff it is Bourbaki quasi-precompact and Bourbaki quasi-complete.

 Special attentions are also given to characterize subsets of the spaces $\ell^p (1\leq p<\infty)$ and $C(X),$ which are Bourbaki quasi-precompact in itself.\par
 The classes of continuous, Cauchy continuous and uniformly continuous functions are well known and there have been extensive investigations on conditions under which one of these notions coincide with another (see \cite{s1}). Naturally one should consider the class of ward continuous functions (i.e. those functions which preserve quasi-Cauchy sequences) and one can ask under which conditions ward continuity coincides with continuity or Cauchy continuity or uniform continuity. In Section 3, a wide class of equivalent conditions for a metric space are obtained where these well known classes of continuous functions individually coincide with ward continuous functions.

 Finally taking cue from \cite{Beer3} where a new class of Lipschitz functions, namely, Cauchy Lipschitz functions  were defined and investigated, in Section 4, we define quasi-Cauchy Lipschitz functions using quasi-Cauchy sequences in the place of Cauchy sequences. As a natural consequence we primarily compare them with other types of Lipschitz functions \cite{Beer1,Beer2,Beer3,g1} and again several coincidence results are proved along with a very interesting observation that every real valued ward continuous function defined on a metric space can be uniformly approximated by real valued quasi-Cauchy Lipschitz functions.\par
 Throughout $\mathbb{R}$ and $\mathbb{N}$ stand for the sets of all real numbers and natural numbers respectively and our topological terminologies and notations are as in the book \cite{e} from where the notions (undefined inside the article) can be found. All spaces in the sequel are metric spaces and all subsets of $\mathbb{R}$ and $\mathbb{R}^2$ are endowed with the usual metric.

\section{ Bourbaki quasi-precompact spaces}
\label{}
    As has already been mentioned, precompact or totally bounded metric spaces are those where every sequence has a Cauchy subsequence. Precompact spaces and its generalizations over the years have always played one of the most important roles in research (in particular compactness related investigations) in metric and uniform spaces as is evident from the vast literature. When quasi-Cauchy sequences come into the picture, it is tempting to define a notion of precompactness in the same way, only replacing Cauchy sequences by quasi-Cauchy sequences but it does not seem to be that much useful a notion because of the wilder nature of pre-Cauchy sequences. However we consider the notion of $\alpha$-boundedness, which was introduced and studied by Tashjain in the context of metric spaces in \cite{tash1} and for uniform spaces in \cite{tash2}. Then we investigate a nice connection in between the notion of $\alpha$-boundedness and quasi-Cauchy sequences, which helps to open a new direction of this concept. Further we show that $\alpha$-bounded spaces play an important role for quasi-Cauchy sequences as precompact spaces play for Cauchy sequences. We start our pursuit of this new analysis  with some notations and recalling some definitions.
\begin{Definition}\cite{bc}
	Let $(X,d)$ be a metric space. A sequence $(x_n)$ in $X$ is called quasi-Cauchy if for every $\varepsilon > 0$ there exists $n_0 \in \mathbb{N}$ such that $d(x_{n+1},x_n)< \varepsilon$ for all $n\geq n_0.$
\end{Definition}
\begin{Definition} (see \cite{caka} for example)
	A function $f:(X,d)\to (Y,\rho)$ is said to be ward continuous if for every quasi-Cauchy sequence $(x_n)$ in $X,$ $(f(x_n))$ is quasi-Cauchy in $Y.$
\end{Definition}
For a metric space $(X,d),$ denote by $B_{\varepsilon}(x),$ the open ball with centre $x\in X$ and radius $\varepsilon >0$ and for any subset $A$ of $X$ and $\varepsilon>0$ we will denote the $\varepsilon$-enlargement of $A$ by $A^{\varepsilon}=\cup \{B_{\varepsilon}(x):x\in A\}=\{y:d(y,A)<\varepsilon\}.$ Furthermore the $\varepsilon$-chainable component of $x\in X$ is defined by $B^{\infty}_{\varepsilon}(x)=\displaystyle{\bigcup_{n\in \mathbb{N}}}B^n_{\varepsilon}(x)$, where $B^1_{\varepsilon}(x)=B_{\varepsilon}(x)$ and for every $n\geq 2,$ $B^n_{\varepsilon}(x)=(B^{n-1}_{\varepsilon}(x))^{\varepsilon}.$
\begin{Definition}
	Let (X,d) be a metric space and $\varepsilon >0$ be given. Then an ordered set of points $\{x_0,x_1,...,x_n\}$ in $X$ satisfying $d(x_{i-1},x_i)<\varepsilon,$ where $i=1,2,...,n$ is said to be an $\varepsilon$-chain of length $n$ from $x_0$ to $x_n.$\par
	Note that $y\in B^n_{\varepsilon}(x)$ iff $x$ and $y$ can be joined by an $\varepsilon$-chain of length $n.$
\end{Definition}
\begin{Definition}\cite{a}
	$(1)$ A metric space $(X,d)$ is called $\varepsilon$-chainable if any two points of $X$ can be joined by an $\varepsilon$-chain, whereas $X$ is called chainable if $X$ is $\varepsilon$-chainable for every $\varepsilon >0.$\par
	$(2)$ A subset $B$ of a metric space $(X,d)$ is said to be Bourbaki bounded (also known as finitely chainable subset of $X$) if for every $\varepsilon >0$ there exist $m\in \mathbb{N}$ and a finite collection of points $p_1,p_2,...,p_k \in X$ such that $B\subseteq \displaystyle{\bigcup_{i=1}^k} B^m_{\varepsilon} (p_i).$
\end{Definition}
Now we consider the definition of $\alpha$-boundedness, a weaker version of precompactness, which we name as ``Bourbaki quasi-precompactness" as it would be clearer later (see Theorem \ref{b}) the important role played by quasi-Cauchy sequences in this notion.
\begin{Definition}
	Let $(X,d)$ be a metric space. Then $B\subseteq X$ is said to be a Bourbaki quasi-precompact subset of $X$ (or sometimes it is called Bourbaki quasi-precompact in $X$) if for every $\varepsilon >0$ there exists a finite collection of points $p_1,p_2,...,p_k \in X$ such that $B\subseteq \displaystyle{\bigcup_{i=1}^k} B^{\infty}_{\varepsilon} (p_i).$
\end{Definition}
 It is clear that totally bounded$\Rightarrow$ Bourbaki bounded$\Rightarrow$ Bourbaki quasi-precompact. But the converse implications are not generally true. Every chainable metric space is Bourbaki quasi-precompact, but it may not be Bourbaki bounded. On the other hand, it is clear that Bourbaki quasi-precompactness is uniform property and the family of Bourbaki quasi-precompact subsets forms a (closed) bornology. One must also keep in mind the subtle difference between ``Bourbaki quasi-precompact in $X$" and ``Bourbaki quasi-precompact in itself" depending on whether the points forming the chains are coming from the concerned subset itself or not. %Here it is important to note that like Bourbaki bounded subset Bourbaki quasi-precompact subset is not an intrinsic property i.e it depends on the ambient space where we are. Note that if $A\subseteq X$ is Bourbaki quasi-precompact in itself then it also Bourbaki quasi-precompact in $X.$ But converse is not generally true.%
  For example note that every subset of a chainable metric space $X$ is Bourbaki quasi-precompact in $X$ but an infinite uniformly discrete subset of $X$ cannot be Bourbaki quasi-precompact in itself. Evidently every quasi-Cauchy sequence is Bourbaki quasi-precompact in itself. The notion of Bourbaki quasi-precompactness is independent with the notion of boundedness. The following examples show that even a bounded Bourbaki quasi-precompact set may not be Bourbaki bounded.
  \begin{Example}
  	The real line with the bounded metric $\widehat{d}=\min\{1,d_{usual}\}$ is bounded Bourbaki quasi-precompact, but not Bourbaki bounded.
  \end{Example}
In the following examples we consider bounded Bourbaki quasi-precompact subsets of $\ell^{\infty}$ and $C(X)$, which are not Bourbaki bounded.
\begin{Example}
	Consider $X_n=\{(1-\frac{k}{n+1})e_n+\frac{k}{n+1}e_{n+1}:k\in \mathbb{Z},0\leq k \leq n+1\},$ where $(e_n)$ is sequence of unit vectors of $\ell^{\infty}.$ Take $X=\displaystyle{\bigcup_{n\in \mathbb{N}}}X_n$ with sup norm of $\ell^{\infty}.$ By suitably arranging the terms of $X$ one can observe that $X$ is a quasi-Cauchy sequence and so Bourbaki quasi-precompact  in itself. Evidently $X$ is bounded. Note that $d(X_i,X_j)=\frac{1}{2}$ for all $i,j \in \mathbb{N}$ with $|i-j|\geq 2.$ So any $\frac{1}{4}$-chain joining $e_{2n}$ and $e_{2m}$ for $m> n$ must meet $X_{2n+1},X_{2n+2},...,X_{2m-1}.$ Consequently the length of this chain must be at least $2(m-n)-1$. Hence $X$ cannot be Bourbaki bounded.
\end{Example}
\begin{Example}
	Consider $X=\{\frac{1}{n}:n\in \mathbb{N}\}\cup \{0\}$ with usual metric of $\mathbb{R}.$ Let us define a function $f^n_k:X\to\mathbb{R}$ by $f^n_k(\frac{1}{n})=1-\frac{k}{n+1}$, $f^n_k(\frac{1}{n+1})=\frac{k}{n+1}$ and $f^n_k(x)=0$ otherwise, where $k\in\mathbb{Z},0\leq k \leq n+1$ $n\in \mathbb{N}$. Take $A=\{f_k^n:k\in\mathbb{Z},0\leq k\leq n+1$ and $n\in\mathbb{N}\}$ with sup norm of $C(X)$. Then proceeding as in the above example $A$ is a bounded subset of $C(X)$ which is Bourbaki quasi-precompact in itself but not Bourbaki bounded.
\end{Example}
Now we are going to present a sequential characterization of Bourbaki quasi-precompact subsets of a space. For that we introduce the notion of Bourbaki quasi-Cauchy sequences.
\begin{Definition}\cite{g2}
	Let $(X,d)$ be a metric space. A sequence $(x_n)$ is said to be Bourbaki-Cauchy in $X$ if for every $\varepsilon >0$ there exist $m\in \mathbb{N}$ and $n_0\in \mathbb{N}$ such that for some $p\in X$ we have $x_n\in  B^m_{\varepsilon}(p)$ for every $n\geq n_0.$
\end{Definition}
\begin{Definition}
Let $(X,d)$ be a metric space. A sequence $(x_n)$ is said to be Bourbaki quasi-Cauchy in $X$ if for every $\varepsilon >0$ there exists $n_0\in \mathbb{N}$ such that for some $p\in X$ we have $x_n\in  B^{\infty}_{\varepsilon}(p)$ for every $n\geq n_0.$
\end{Definition}
Clearly every subsequence of a Bourbaki quasi-Cauchy sequence is Bourbaki quasi-Cauchy in the underlying space. We will say that a sequence $(x_n)$ has a Bourbaki quasi-Cauchy subsequence in $X$ if $(x_n)$ has a subsequence which is Bourbaki quasi-Cauchy in $X.$ The reason behind inclusion of the term ``quasi-Cauchy" in Definition 2.7 can be understood from the following result.
\begin{Theorem}\label{a}
	Let $(X,d)$ be a metric space. A sequence $(x_n)$ is Bourbaki quasi-Cauchy in $X$ iff $(x_n)$ is a subsequence of some quasi-Cauchy sequence of $X.$
\end{Theorem}
\begin{proof}
Suppose that $(x_n)$ is a subsequence of a quasi-Cauchy sequence $(z_k).$ Let $\varepsilon >0$ be given. Then there exists $k_0\in \mathbb{N}$ such that $d(z_{k+1},z_k)< \varepsilon$ $\forall k \geq k_0$ which implies that $z_k\in B_{\varepsilon}^{\infty}(z_{k_0})$ for all $k\geq k_0.$ Consequently $x_n\in B^{\infty}_{\varepsilon}(z_{k_0})$ for all but finitely many $n$.\par
For the converse, let $(x_n)$ be a Bourbaki quasi-Cauchy sequence in $X.$ Now for each $k\in \mathbb{N}$ there exists $n_k$ such that $x_i\in B^{\infty}_{\frac{1}{k}}(x_{n_k})$ for all $i\geq n_k.$ Equivalently we can say that $x_{i+1}$ and $x_i$ can be joined by a $\frac{1}{k}$-chain for each $i\geq n_k$ i.e. there exists a finite collection of points $x_i=p_0^{i,k},p_1^{i,k},....,p_{r_i}^{i,k}=x_{i+1}$ with the property that $d(p^{i,k}_{j+1},p^{i,k}_j)<\frac{1}{k}$ for all $j=0,1,..,r_i.$ Now consider the sequence
\begin{center}
$\{x_1,..,x_{n_1},p_1^{n_1,1},p_2^{n_1,1},...,x_{n_1+1},....,x_{n_2},p_1^{n_2,\frac{1}{2}},...x_{n_k},p_1^{n_k,\frac{1}{k}},.......\}.$
\end{center}
Clearly the above sequence is the required quasi-Cauchy sequence.
\end{proof}
\begin{Corollary}
	Every subsequence of a quasi-Cauchy sequence $(x_n)$ is Bourbaki quasi-Cauchy in $(x_n)$.
\end{Corollary}
In the following, we present a sequential characterization of Bourbaki quasi-precompact sets in line of the classical result that a space is precompact iff every sequence has a Cauchy subsequence.
\begin{Theorem}\label{b}
	Let $(X,d)$ be a metric space. Then a non-void subset $A$ of $X$ is Bourbaki quasi-precompact in $X$ iff every sequence in $A$ has a Bourbaki quasi-Cauchy subsequence in $X.$
\end{Theorem}
\begin{proof}
	Suppose that $\phi \neq A\subseteq X$ is Bourbaki quasi-precompact in $X$ and $(x_n)$ is a sequence in $A.$ Without any loss of generality let us assume that all $x_n$'s are distinct. Consider all $1$-enlargements $B^{\infty}_1(y)$ where $y\in X.$ Clearly there exists $y_1\in X$ such that $B_1^{\infty}(y_1)$ must contain infinitely many terms of $(x_n).$ Take $I_1=\{n:x_n\in B_1^{\infty}(y_1)\}.$ Similarly there exists $y_2\in X$ such that $B_{\frac{1}{2}}^{\infty}(y_2)$ contains infinite number of terms of $(x_n)_{n\in I_1},$ then take $I_2=I_1\cap \{n:x_n\in B_{\frac{1}{2}}^{\infty}(y_2)\}$ and so on. Continuing this process we obtain a decreasing sequence $(I_k)$ of infinite subsets of $\mathbb{N}$ where $I_{k+1}=I_k \cap \{n:x_n\in B_{\frac{1}{k+1}}^{\infty}(y_{k+1})\}$ for each $k\in \mathbb{N}.$ Choose an increasing sequence $(n_k)$ of natural numbers with $n_k\in I_k$. We claim that $(x_{n_k})$ is a Bourbaki quasi-Cauchy sequence in $X.$ This is true because given $\varepsilon >0,$ one can first choose $k_0\in \mathbb{N}$ such that $\frac{1}{k_0}< \varepsilon$ and then from the construction of $I_k$ it can be concluded that $x_{n_k}\in B^{\infty}_{\frac{1}{k_0}}(y_{k_0})$ for all $k\geq k_0.$\par
	Conversely, suppose on the contrary that the given condition holds, but $A$ is not Bourbaki quasi-precompact in $X.$ Then there exists an $\varepsilon >0$ for which we can choose $x_1,x_2\in A$ such that $x_2\notin B^{\infty}_{\varepsilon}(x_1).$ Again we can choose $x_3\in A$ such that $x_3\notin B^{\infty}_{\varepsilon}(x_1) \cup B^{\infty}_{\varepsilon}(x_2).$ Continuing in this process a sequence $(x_k)$ in $A$ is obtained which has the property that $x_{k+1}\notin\displaystyle{\bigcup_{i=1}^k} B^{\infty}_{\varepsilon}(x_i).$ Obviously $(x_n)$ has no Bourbaki quasi-Cauchy subsequence in $X$. This contradicts the given condition. Hence $A$ is Bourbaki quasi-precompact in $X.$
\end{proof}
\begin{Corollary}
	A metric space $X$ is Bourbaki quasi-precompact iff every sequence has a Bourbaki quasi-Cauchy subsequence.
\end{Corollary}

\begin{Theorem}
	Let $(X,d), (Y,\rho)$ be two metric spaces and $f:(X,d)\to (Y,\rho)$ be a ward continuous function. Then for any Bourbaki quasi-precompact subset $A$ of $X$, $f(A)$ is Bourbaki quasi-precompact in $f(X).$
\end{Theorem}
\begin{proof}
	The proof follows from Theorems \ref{a} and \ref{b} and the fact that ward continuous functions preserve quasi-Cauchy condition.
\end{proof}
\begin{Theorem}
	Let $(X,d)$ be a metric space. Then $A\subseteq X$ is Bourbaki quasi-precompact in $X$ iff $\overline{A}$ is Bourbaki quasi-precompact in $X.$
\end{Theorem}
\begin{proof}
The proof easily follows by using Proposition 17 of \cite{g3}.
\end{proof}
Now we introduce a certain version of completeness and is an intermediate property between compactness and completeness.
\begin{Definition}
	Let $(X,d)$ be a metric space. Then $A\subseteq X$ is called Bourbaki quasi-complete subset of $X$ (or sometimes it is called Bourbaki quasi-complete in $X$) if every sequence in $A$ which is Bourbaki quasi-Cauchy in $X,$ has a cluster point in $A$. Whereas $A$ is called Bourbaki quasi-complete space if it is Bourbaki quasi-complete in itself.
\end{Definition}
\begin{Remark}
	Note that if $A\subseteq X$ is Bourbaki quasi-complete in $X$ then $A$ is also Bourbaki quasi-complete in itself as every Bourbaki quasi-Cauchy sequence in $A$ is also Bourbaki quasi-Cauchy in $X.$ But converse is not generally true. As an example let us take the sequence $(n)$ which is Bourbaki quasi-Cauchy in $\mathbb{R}$ but not in $\mathbb{N}.$
\end{Remark}
\begin{Example}
	$\mathbb{N}$ is a Bourbaki quasi-complete space but not compact and $\{\sqrt{n}:n\in \mathbb{N}\}$ is complete but not a Bourbaki quasi-complete space. Note that $\mathbb{N}$ is not Bourbaki quasi-complete in $\mathbb{R}.$
\end{Example}
Bourbaki quasi-completeness is stronger than completeness and Bourbaki quasi-precompactness is weaker than precompactness. Interestingly their combinations back to compactness. The next result is in that direction. We are now in a position to present a simple but interesting new characterization of compactness in metric structures.
\begin{Theorem}
	A metric space $X$ is compact iff it is Bourbaki quasi-precompact and Bourbaki quasi-complete.
\end{Theorem}
\begin{proof}
The necessity of the conditions follow from classical observations and Theorem \ref{b}. For the converse part, let $(x_n)$ be a sequence in $X.$ From Theorem \ref{b}, $(x_n)$ has a Bourbaki quasi-Cauchy subsequence and so it has a cluster point by the given condition. Hence $X$ is compact.
\end{proof}
For a subset of $\mathbb{R}$, we can actually characterize compactness with the help of Bourbaki quasi-precompactness along with a weaker condition, known as weakly G-completeness \cite{vg}.
\begin{Theorem}
	Let $X\subseteq \mathbb{R}$ be endowed with the usual metric of $\mathbb{R}.$ Then $X$ is compact iff it is Bourbaki quasi-precompact in itself and every quasi-Cauchy sequence in $X$ has a cluster point in $X$.
\end{Theorem}
\begin{proof}
	One part is obvious. Conversely suppose that $X$ satisfies the mentioned property and it is Bourbaki quasi-precompact in itself. As closedness of $X$ in $\mathbb{R}$ is quite obvious, one only needs to prove that $X$ is bounded. Suppose on the contrary that $X$ is not bounded. Then we can choose a sequence $(x_n)$ which satisfies the property that $x_1 > 1$ and $x_{n+1} > x_n+1$ for all $n\in \mathbb{N}.$ Now as $X$ is a Bourbaki quasi-precompact space,  by passing to a subsequence we can conclude that $(x_n)$ is Bourbaki quasi-Cauchy in $X.$ So there is a quasi-Cauchy sequence $(z_n)$ in $X$ such that $z_{r_n}=x_n$ for some subsequence $(z_{r_n})$ of $(z_n).$ Now consider $G_n=\{z_i: r_n<i <r_{n+1}\ and\ z_i\in (z_{r_n},z_{r_{n+1}})\}$ for each $n\in \mathbb{N}.$ Clearly each $G_n$ is non-empty for all but finitely many $n$. We can write $G_n$ as $\{z_1^n < z_2^n <....< z_{p_n}^n\}.$ Now consider the following increasing sequence
	\begin{center}
		$\{z_{r_1},z^1_1,z^1_2,...,z^1_{p_1},z_{r_2},z^2_1,...,z_{r_3},...,z_{r_n},z^n_1,...,z^n_{p_n},z_{r_{n+1}},.....\}.$
	\end{center}
Observe that it still remains a quasi-Cauchy sequence in $X,$ but it has no cluster point. This contradicts the given condition. Hence $X$ must be bounded and consequently $X$ must be compact.
\end{proof}
 However the following example shows that the above characterization is not generally true for an arbitrary metric space. Also it shows that Bourbaki quasi-completeness is strictly stronger than weakly G-completeness.
\begin{Example}
	Consider $X=\{re_n:n\in \mathbb{N},r\in [0,1]\},$ where $\{e_n:n\in \mathbb{N}\}$ is the set of all unit vectors of $\ell^{\infty}$, endowed with sup norm of $\ell^{\infty}.$ $X$ is a Bourbaki quasi-precompact space as it is chainable. Let $(x_n)$ be a quasi-Cauchy sequence in $X.$ Without any loss of generality we can assume that $x_n$'s are distinct. Choose $X_k=\{re_k:r\in [0,1]\}.$ If $(x_n)$ intersects only finitely many of $X_k,$ then from the compactness of $X_k$'s it follows that $(x_n)$ has a cluster point. Now if $(x_n)$ intersects infinitely many of $X_k$, then we can choose a subsequence $(x_{n_p})$ of $(x_n)$ and a sub collection $\{X_{k_p}:p\in \mathbb{N}\}$ of $\{X_k:k\in \mathbb{N}\}$ such that $x_{n_p}\in X_{k_p}$ but $x_{n_p+1}\notin X_{k_p}.$ Let $r_p$ and $s_p$ be the non-zero terms of $x_{n_p}$ and $x_{n_p+1}$ respectively. Then we have $||x_{n_p+1}-x_{n_p}||_{\infty}=\max\{s_p,r_p\}\to 0.$ Hence $x_{n_p}\to 0$. But $X$ is not compact. Note that $(e_n)$ is Bourbaki quasi-Cauchy in $X$ without any cluster point, in $X$.
\end{Example}
 There is a nice characterization of precompact subsets of a metric space $X$ making use of uniformly discrete subsets in \cite{Ulydis} as follows ``A space $X$ is precompact iff every uniformly discrete subset is finite". We will now establish such an analogous result for Bourbaki quasi-precompact spaces. In order to serve our purpose we introduce the following versions of discrete space and uniformly discrete space.
 \begin{Definition}
 	Let $(X,d)$ be a metric space.\par
 	$(1)$ $A\subseteq X$ is called chain discrete in $X$ if for each $x\in A$ there exists $\delta_x >0$ such that $B^{\infty}_{\delta_x}(x)\cap [A\setminus\{x\}] = \phi.$\par
 	$(2)$ $A\subseteq X$ is called uniformly chain discrete in $X$ if there exists $\delta >0$ such that $B^{\infty}_{\delta}(x)\cap [A\setminus\{x\}] = \phi$ for all $x\in A.$
 \end{Definition}
Clearly chain discrete $\Rightarrow$ discrete and uniformly chain discrete $\Rightarrow$ uniformly discrete. The converse implications are not generally true as $\mathbb{N}$ is discrete but not chain discrete in $\mathbb{R}.$ But it is uniformly chain discrete in itself.

\begin{Theorem}
	A metric space $X$ is Bourbaki quasi-precompact iff every uniformly chain discrete subset in $X$ is finite.
\end{Theorem}
\begin{proof}
	Let $X$ be Bourbaki quasi-precompact. On the contrary suppose that there is an infinite uniformly chain discrete subset $A$ of $X$. Then there exist a sequence $(x_n)$ in $A$ and $\delta>0$ such that $x_i\notin B^{\infty}_{\delta}(x_j)$ for all $i\neq j.$ Obviously this contradicts that $X$ is Bourbaki quasi-precompact. Hence $A$ must be finite.\par
	Conversely, let $(x_n)$ be a sequence in $X.$ Without any loss of generality we can assume that $x_n$'s are distinct. If $(x_n)$ has no Bourbaki quasi-Cauchy subsequence then there exists an $\varepsilon >0$ and passing to a subsequence we can conclude that $x_i\notin B^{\infty}_{\varepsilon}(x_j)$ whenever $i\neq j.$ Hence $(x_n)$ is an infinite uniformly chain discrete subset of $X,$ which contradicts the given assumption.
\end{proof}
We end this section with two results  presenting typical characterizations of subsets of $C(X)$ and $\ell^p$ $(1\leq p < \infty),$ which are Bourbaki quasi-precompact in itself.

The following notion ``equi-chain continuity"  (an weaker version of the notion of equi-continuity) will come handy for the aforementioned  characterization in $C(X).$
\begin{Definition}
	Let $(X,d)$ be a metric space. $A\subseteq C(X)$ is called equi-chain continuous if for each $\varepsilon >0$ and for every $x\in X$ there is $\delta_x >0$ such that for each $f\in A$ there exists $g_f^{\varepsilon}\in A$, which satisfies the property that $d(x,y)< \delta_x \Rightarrow |g^{\varepsilon}_f(x)-g_f^{\varepsilon}(y)| < \varepsilon$ and furthermore $f, g_f^{\varepsilon}$ can be joined by an $\varepsilon$-chain in $A.$
\end{Definition}
Clearly equi-continuity implies equi-chain continuity. But the converse is not generally true. The following example illustrates this fact.
\begin{Example}
	Define $f_n:[0,1]\to [0,1]$ by $f_n(x)=nx$ if $x\leq \frac{1}{n}$ and $f_n(x)=1$ if $x\geq \frac{1}{n}.$ Clearly $f_n\in C([0,1])$ for each $n\in \mathbb{N}.$ Consider $A=\{f_n:n\in \mathbb{N}\}$ with sup norm of $C([0,1]).$ It is easy to check that $(f_n)$ is a quasi Cauchy sequence in $C([0,1]).$ Clearly $A$ is equi-chain continuous. But $A$ is not equi-continuous at $x=0.$ To see that, take $\varepsilon=\frac{1}{2}$ and observe that for each $n\in \mathbb{N}$, $|f_n(\frac{1}{n})-f_n(0)|=1>\frac{1}{2}.$ Furthermore $A$ is not precompact as $(f_{2^k})_{k\in \mathbb{N}}$ has no Cauchy subsequence.
\end{Example}
\begin{Theorem}
	Let $(X,d)$ be a compact metric space and $A$ be a bounded subset of $C(X)$ with sup norm. Then $A$ is Bourbaki quasi-precompact in itself iff $A$ is equi-chain continuous.
\end{Theorem}
\begin{proof}
	First let $A\subseteq C(X)$ be Bourbaki quasi-precompact in itself and let $\varepsilon >0$ be given. Then there exist $f_1,f_2,...,f_k \in A$ such that $A\subseteq \displaystyle{\bigcup_{i=1}^k} B^{\infty}_{\varepsilon}(f_i)$ in $A.$ Now for each $x\in X$ there exists $\delta_x^i > 0$ such that $d(x,y)< \delta_x^i \Rightarrow |f_i(x)-f_i(y)|< \varepsilon$ for each $i=1,2,..,k.$ Take $\delta_x=\min\{\delta_x^i:i=1,2,...,k\}.$ Consequently we can conclude that $d(x,y)< \delta_x \Rightarrow |f_i(x)-f_i(y)|< \varepsilon$ for each $i=1,2,..,k$ and any $f\in A$ can be joined to some $f_i$ by an $\varepsilon$-chain in $A.$  This shows that $A$ is equi-chain continuous.\par
	Conversely, let $A$ be equi-chain continuous and let $\varepsilon >0$ be given. From the  equi-chain continuity condition, for each $x\in X$ one can find a $\delta_x >0$ such that for each $f\in A$ there exists $g_f^{\frac{\varepsilon}{3}}\in A$ which satisfies the property that $d(x,y)< \delta_x \Rightarrow |g^{\frac{\varepsilon}{3}}_f(x)-g_f^{\frac{\varepsilon}{3}}(y)| < \frac{\varepsilon}{3}$ and moreover $f, g_f^{\frac{\varepsilon}{3}}$ can be joined by an $\frac{\varepsilon}{3}$-chain in $A.$ As $X$ is compact, we can conclude that $\{B_{\delta_x}(x):x\in X\}$ has a finite subcover. Suppose that $X\subseteq \displaystyle{\bigcup_{i=1}^k} B_{\delta_{x_i}}(x_i).$ Now define a map $\phi:A\to \mathbb{R}^k$ by $\phi(f)=(g^{\frac{\varepsilon}{3}}_f(x_1),g^{\frac{\varepsilon}{3}}_f(x_2),....,g^{\frac{\varepsilon}{3}}_f(x_k)).$ To make $\phi$ well defined we select one $g^{\frac{\varepsilon}{3}}_f$ for each $f.$ As $A$ is bounded, $\phi(A)$ is precompact and so there are $f_1,f_2,...,f_m \in A$ such that $\phi(A)\subseteq \displaystyle{\bigcup_{j=1}^m} B_{\frac{\varepsilon}{6}}(\phi(f_j))$ in $\mathbb{R}^k.$ Write $B_{\frac{\varepsilon}{6}}(\phi(f_j))=U_j.$ We claim that for any $f,h \in \phi^{-1}(U_j)$ where $j=1,...,m,$ $f$ and $h$ can be joined by an $\varepsilon$-chain in $A.$ Note that $f,h \in \phi^{-1}(U_j)\Rightarrow |g^{\frac{\varepsilon}{3}}_f(x_i)- g^{\frac{\varepsilon}{3}}_h(x_i)|< \frac{\varepsilon}{6}+\frac{\varepsilon}{6} =\frac{\varepsilon}{3}$ for every $i=1,...,k.$ Now for any $x\in X$ there exists $x_i$ such that $x\in B_{\delta_{x_i}}(x_i).$ Consequently it follows that $|g^{\frac{\varepsilon}{3}}_f(x)- g^{\frac{\varepsilon}{3}}_h(x)| \leq |g^{\frac{\varepsilon}{3}}_f(x)-g^{\frac{\varepsilon}{3}}_f(x_i)|+|g^{\frac{\varepsilon}{3}}_f(x_i)-g^{\frac{\varepsilon}{3}}_h(x_i)|+|g^{\frac{\varepsilon}{3}}_h(x_i)-g^{\frac{\varepsilon}{3}}_h(x)|< \frac{\varepsilon}{3}+\frac{\varepsilon}{3}+\frac{\varepsilon}{3}=\varepsilon.$ Hence $||g^{\frac{\varepsilon}{3}}_f - g^{\frac{\varepsilon}{3}}_h||< \varepsilon.$ Moreover $f$ and $h$ can be joined by an $\frac{\varepsilon}{3}$-chain in $A$ with $g^{\frac{\varepsilon}{3}}_f$ and $g^{\frac{\varepsilon}{3}}_h$ respectively. Now it is easy to check that $A\subseteq \displaystyle{\bigcup_{j=1}^m} \phi^{-1}(U_j).$ Hence $A$ is Bourbaki quasi-precompact in itself.
\end{proof}
We end with the following result for the space $\ell^p$.
\begin{Theorem}
	Let $A$ be a bounded subset of $\ell^p(1 \leq p < \infty)$ with $p$-norm. Then $A$ is Bourbaki quasi-precompact in itself iff for each $\varepsilon >0$ there exists $n_0\in \mathbb{N}$ such that for each $x\in A$ there exists $y_x^{\varepsilon} \in A,$ satisfying the following property
	
	(*) $x$ can be joined by an $\varepsilon$-chain to $y_x^{\varepsilon}$ in $A$ and $\displaystyle{\sum_{i> n_0}} |y^{\varepsilon}_{x,i}|^p < \varepsilon^p$ where $y^{\varepsilon}_x=(y^{\varepsilon}_{x,i}).$
\end{Theorem}
\begin{proof}
	First let $A\subseteq \ell^p$ be Bourbaki quasi-precompact in itself and let $\varepsilon >0$ be given. Then there exist $x_1,x_2,...,x_k \in A$ such that $A\subseteq \displaystyle{\bigcup_{i=1}^k} B^{\infty}_{\varepsilon}(x_i)$ in $A.$ Note that there exists $n_i\in \mathbb{N}$ such that $\displaystyle{\sum_{j> n_i}} |x_{i,j}|^p < \varepsilon^p$ where $x_i=(x_{i,j}).$ Take $n_0=\max\{n_1,...,n_k\}.$ Subsequently $\displaystyle{\sum_{j> n_0}} |x_{i,j}|^p < \varepsilon^p$ for each $i=1,..,k$ and each $x\in A$ can be joined to some $x_i$ by an $\varepsilon$-chain in $A$. \par
	Conversely, suppose that $A\subseteq \ell^p$ satisfies the condition (*) and let $\varepsilon >0$ be given. Then there exists $n_0\in \mathbb{N}$ such that for each $x\in A$ one can find a $y_x^{\frac{\varepsilon}{3}} \in A,$ so that $x$ can be joined by an $\frac{\varepsilon}{3}$-chain to $y_x^{\frac{\varepsilon}{3}}$ in $A$ and furthermore $\displaystyle{\sum_{i> n_0}} |y^{\frac{\varepsilon}{3}}_{x,i}|^p < (\frac{\varepsilon}{3})^p$ where $y^{\frac{\varepsilon}{3}}_x=(y^{\frac{\varepsilon}{3}}_{x,i}).$ Now define a map $\phi:A\to \mathbb{R}^{n_0}$ by $\phi(x)=(y^{\frac{\varepsilon}{3}}_{x,1},y^{\frac{\varepsilon}{3}}_{x,2},....,y^{\frac{\varepsilon}{3}}_{x,n_0}).$ We can make $\phi$ well defined by selecting only one $y^{\frac{\varepsilon}{3}}_x$ for each $x.$ As $A$ is bounded, $\phi(A)$ is precompact and consequently there are finite number of points $x_1,x_2,...,x_m \in A$ such that $\phi(A)\subseteq \displaystyle{\bigcup_{j=1}^m} B_{\delta}(\phi(x_j))$ in $\mathbb{R}^{n_0},$ where $n_0\delta^p < (\frac{\varepsilon}{6})^p.$ Write $B_{\delta}(\phi(x_j))=U_j.$ Note that for any $x,z \in \phi^{-1}(U_j)$ we can obtain that $||y^{\frac{\varepsilon}{3}}_x - y^{\frac{\varepsilon}{3}}_z||_p \leq (\displaystyle{\sum_{i=1}^{n_0}} |y^{\frac{\varepsilon}{3}}_{x,i}- y^{\frac{\varepsilon}{3}}_{z,i}|^p)^{\frac{1}{p}} + (\displaystyle{\sum_{i> n_0}} |y^{\frac{\varepsilon}{3}}_{x,i}- y^{\frac{\varepsilon}{3}}_{z,i}|^p)^{\frac{1}{p}}< \frac{\varepsilon}{3}+\frac{\varepsilon}{3}+\frac{\varepsilon}{3} = \varepsilon.$ Moreover $x$ and $z$ can be joined by an $\frac{\varepsilon}{3}$-chain in $A$ with $y^{\frac{\varepsilon}{3}}_x$ and $y^{\frac{\varepsilon}{3}}_z$ respectively. Now it is easy to check that $A\subseteq \displaystyle{\bigcup_{j=1}^m} \phi^{-1}(U_j).$ Therefore $A$ is Bourbaki quasi-precompact in itself.
\end{proof}

\section{Certain observations on ward continuous functions}
\label{}
In this section our main interest are functions which preserve quasi-Cauchy sequences i.e. ward continuous functions.

 First we recall some basic definitions before proceeding to our  main results.
\begin{itemize}
\item Let $(X,d)$ be a metric space. Then $I(x)=d(x,\{x\}^c),$ which is the degree of isolation of a point $x\in X.$ $I(x)=0$ iff $x\in X^{'},$ where $X^{'}$ is the set of all limit point of $X$ \cite{a}.
\item A sequence $(x_n)$ is called pseudo-Cauchy if for all $\varepsilon >0$ and for all $n\in \mathbb{N},$ there exist $j,k \in \mathbb{N}$ such that $j\neq k, j,k>n$ and $d(x_j,x_k)< \varepsilon$ \cite{s1}.
\item A space $X$ is called $UC$ (Atsuji) iff every real valued continuous function defined on $X$ is uniformly continuous \cite{a}.
\item Let $(X,d)$ be a metric space. Then two subsets $\phi\neq A,B$ of $X$ are called Cauchy separated if $A,B$ do not have any common Cauchy sequence i.e. for each Cauchy sequence $(x_n)$ in $X$ the sets, $\{n : x_n \in A\}$ and $\{n : x_n \in B\}$ cannot be infinite simultaneously \cite{PDSPNA1}.
\item Let $\phi \neq A, B \subset X$. Then $A,B$ are said to be connected through a quasi-Cauchy sequence if there exists a quasi-Cauchy sequence $(x_n)$ in $X$ such that $x_{n_k} \in A$ and $x_{{n_k}+1} \in B$ for some subsequence $(x_{n_k})$ of $(x_n)$ \cite{PDSPNA1}.
\end{itemize}
We start by noting that in line of the following classical results:\par
$(1)$ In a totally bounded space every real valued Cauchy continuous function is uniformly continuous.\par
$(2)$ $f:X\to \mathbb{R}$ Cauchy continuous iff $f|_A$ is uniformly continuous for every totally bounded subset $A$ of $X.$\par

One can obtain similar observations regarding ward continuous functions.
\begin{Theorem}\label{c}
Let $X$ be a Bourbaki quasi-precompact space. Then every real valued ward continuous function defined on $X$ is uniformly continuous.
\end{Theorem}
\begin{Theorem}
Let $(X,d)$ be a metric space. Then a function $f:X\to \mathbb{R}$ is ward continuous iff restriction of $f$ on $A$ is uniformly continuous for every Bourbaki quasi-precompact subset $A$ of $X.$
\end{Theorem}
Further one can use ward continuous functions to characterize Bourbaki quasi-precompact spaces as follows.
\begin{Theorem}(cf. \cite{Hueber})
A space $X$ is Bourbaki quasi-precompact iff every real valued ward continuous function defined on $X$ is uniformly continuous and for every $\varepsilon >0,$ $\{x:I(x)>\varepsilon\}$ is finite.
\end{Theorem}
\begin{proof}
Let $X$ be a Bourbaki quasi-precompact space. Then by Theorem \ref{c} every real valued ward continuous function defined on $X$ is uniformly continuous. For the next part, suppose on the contrary that there exists $\varepsilon >0$ such that $E=\{x:I(x)> \varepsilon \}$ is infinite. Then from Bourbaki quasi-precompactness of $X$ one can find finitely many points $x_1,x_2,...,x_k$ such that $X=\displaystyle{\bigcup_{i=1}^k} B^{\infty}(x_i,\varepsilon).$ Evidently there exists some $y\in E$ such that $y\neq x_i$ for all $i=1,...,k.$ Let $y\in B^{\infty}_{\varepsilon}(x_{i_0})$ for some $i_0\in \{1,..,k\}.$ This implies existence of finitely many distinct points $y=p_0,p_1,...,p_n=x_{i_0}$ such that $d(p_{j-1},p_j)<\varepsilon $ for all $j=1,...,n.$ As $d(y,p_1)<\varepsilon,$ $I(y)<\varepsilon.$ Which contradicts the fact that $y\in E.$ Hence $E$ is finite.\par
Conversely let $(x_n)$ be a sequence in $X.$ By the given condition $\{n:I(x_n) >\frac{\varepsilon}{2}\}$ is finite for every $\varepsilon >0,$ which implies that $I(x_n)\to 0.$ Consequently from Theorem \ref{d} we can conclude that $(x_n)$ has a Bourbaki quasi-Cauchy subsequence in $X$ and hence $X$ is Bourbaki quasi-precompact.
\end{proof}

In \cite{s1}, a wide class of equivalent conditions was presented characterizing spaces where Cauchy continuity agrees with uniform continuity. Taking quasi-Cauchy sequences and corresponding notion of continuity, i.e., ward continuity into the picture, it is natural to ask for conditions to characterize spaces where Cauchy continuity would agree with ward continuity or ward continuity would agree with uniform continuity etc. We precisely deal with such questions in this section. As we will see from the next few results, some of these conditions are expectedly analogous to existing ones \cite{s1}, but there are also certain conditions which were never brought up in the literature. Following is the simplest observation in this line.
\begin{Theorem}
Let $(X,d)$ and $(Y,\rho)$ be two metric spaces. Then the following statements are equivalent.\par
$(1)$ $X$ is Bourbaki quasi-complete.\par
$(2)$ Every subsequence of a quasi-Cauchy sequence in $X$ has a cluster point in $X$.\par
$(3)$ Every real valued continuous function defined on $X$ is ward continuous.
\end{Theorem}
\begin{proof}
The equivalence of $(1)$ and $(2)$ follows from Theorem \ref{a} and equivalence of $(2)$ and $(3)$ follows from Theorem 3.3 of \cite{PDSPNA1}.
\end{proof}
Our next result concerns with the coincidence of ward continuity with Cauchy continuity.
\begin{Theorem}
Let $(X,d)$ and $(Y,\rho)$ be two metric spaces and let $(\widehat{X},\widehat{d})$ denote the completion of $X.$ Then the following conditions are equivalent.\par
$(1)$ Every Cauchy continuous function $f:(X,d)\to (Y,\rho)$ is ward continuous.\par
$(2)$ Every real valued Cauchy continuous function defined on $X$ is ward continuous.\par
$(3)$ Every Bourbaki quasi-Cauchy sequence in $X$ has a Cauchy subsequence.\par
$(4)$ $(\widehat{X},\widehat{d})$ is Bourbaki quasi-complete.\par
$(5)$ Any two Cauchy separated sets cannot be connected through a quasi-Cauchy sequence.\par
$(6)$ For a complete subset $A$ and a closed subset $B$ of $X$ with $A,B\neq \phi$ and $A\cap B=\phi$, $A$ and $B$ cannot be connected through a quasi-Cauchy sequence.
\end{Theorem}
\begin{proof}
$(1)\Rightarrow (2)$ straightforward.\par
$(2)\Rightarrow (3)$ On the contrary suppose that there is a Bourbaki quasi-Cauchy sequence $(x_k)$ which has no Cauchy subsequence. Theorem \ref{a} already assures the existence of a quasi-Cauchy sequence $(z_n)$ such that $z_{n_k}=x_k$ for some subsequence $(z_{n_k})$ of $(z_n).$ Clearly $(z_{n_k+1})$ has no Cauchy subsequence. Now we define a function $f:(z_{n_k})\cup (z_{n_k+1})\to \mathbb{R}$ by $f(z_{n_k})=0$ and $f(z_{n_k+1})=1$ for every $k\in \mathbb{N}.$ As $f$ is Cauchy continuous and $(z_{n_k})\cup (z_{n_k+1})$ is closed in $\widehat{X}$, by Tietze extension theorem $f$ can be extended to a continuous function $\widehat{f}:\widehat{X}\to \mathbb{R}$ such that the restriction of $\widehat{f}$ on $(z_{n_k})\cup (z_{n_k+1})$ is equal to $f.$ Observe that $\widehat{f}|_X$ is Cauchy continuous but it is not ward continuous, which contradicts the given assumption.\par
$(3)\Rightarrow (4)$ Let $(x_k)$ be a Bourbaki quasi-Cauchy sequence in $\widehat{X}.$ Then there exists a quasi-Cauchy sequence $(y_n)$ in $\widehat{X}$ which contains $(x_k)$ as a subsequence. In fact there exists a quasi-Cauchy sequence $(z_n)$ in $X$ such that $\widehat{d}(y_n,z_n)\to 0.$ From the given condition it follows that each subsequence of $(z_n)$ has a Cauchy subsequence. This in turn implies that $(x_k)$ has a Cauchy subsequence and consequently $(x_k)$ has a cluster point in $\widehat{X}.$\par
$(4)\Rightarrow (5)$ Let $A$ and $B$ be two Cauchy separated subsets of $X.$ If possible suppose that they can be connected through a quasi-Cauchy sequence. Then there exist a quasi-Cauchy sequence $(x_k)$ in $X$ and a subsequence $(x_{k_p})$ of $(x_k)$ such that $x_{k_p}\in A$ and $x_{k_p+1}\in B$ for all $p\in \mathbb{N}.$ From the given condition and passing to a subsequence we can conclude that $(x_{k_p})$ is Cauchy and consequently $(x_{k_p+1})$ is also Cauchy. This contradicts the fact that $A,B$ are Cauchy separated.\par
$(5)\Rightarrow (6)$ Let $A$ and $B$ be two non-empty subsets of $X$ where $A$ is complete and $B$ is closed with $A\cap B=\phi.$ It is obvious that $A$ and $B$ are Cauchy separated and hence from $(5)$ we can conclude that they cannot be connected through a quasi-Cauchy sequence.\par
$(6)\Rightarrow (1)$ Let $f:(X,d)\to (Y,\rho)$ be Cauchy continuous. If possible suppose that $f$ is not ward continuous. Then there exists a quasi-Cauchy sequence $(x_n)$ in $X$ but $(f(x_n))$ is not quasi-Cauchy in $Y.$ Consequently one can find a suitable $\varepsilon >0$ and a subsequence $(x_{n_k})$ of $(x_n)$ such that $\rho(f(x_{n_k+1}),f(x_{n_k}))\geq \varepsilon$ for all $k\in\mathbb{N}$. From Cauchy continuity of $f$ we can then conclude that $(x_{n_k})$ has no Cauchy subsequence. Now taking $A=\{x_{n_k+1}:k\in \mathbb{N}\}$ and $B=\{x_{n_k}:k\in \mathbb{N}\}$ we obtain two non-empty complete subsets of $X,$ which are connected through a quasi-Cauchy sequence. This contradicts the given assumption and hence $f$ must be ward continuous.
\end{proof}
Next we focus on the coincidence of ward continuity with uniform continuity.
\begin{Theorem}\label{d}
Let $(X,d)$ be a metric space and $X^{'}$ denote (as usual) the set of all accumulation points in $(X,d).$  Then the following conditions are equivalent.\par
$(1)$ Every real valued ward continuous function defined on $X$ is uniformly continuous.\par
$(2)$ Every pseudo-Cauchy sequence with distinct terms in $X$ has a Bourbaki quasi-Cauchy subsequence in $X.$\par
$(3)$ Every sequence $(x_n)$ in $X$ satisfying $\displaystyle{\lim_{n\to \infty}}I(x_n)=0$ has a Bourbaki quasi-Cauchy subsequence in $X.$\par
$(4)$ $X^{'}$ is Bourbaki quasi-precompact in $X$ and  any infinite Bourbaki quasi-complete subset $U$ of $X$, disjoint from $X^{'}$, is uniformly discrete.% provided  every Bourbaki quasi-Cauchy sequence of $X$ contained in $U$ has a cluster point in $U.$ %(it means that every Bourbaki quasi-Cauchy sequence $(x_n)$ in $X$ with $x_n\in U\ \forall\ n\in \mathbb{N}$, has a cluster point in $U$. Note that here $(x_n)$ may not be Bourbaki quasi-Cauchy in $U$.)
\par
$(5)$ $X^{'}$ is Bourbaki quasi-precompact in $X$ and any infinite Bourbaki quasi-complete subset $U$ of $X,$ disjoint from $X^{'}$, satisfies that $\inf\{I(x):x\in U\}>0$.% provided every Bourbaki quasi-Cauchy sequence of $X$ contained in $U$ has a cluster point in $U.$ %(it means that every Bourbaki quasi-Cauchy sequence $(x_n)$ in $X$ with $x_n\in U\ \forall\ n\in \mathbb{N}$, has a cluster point in $U$. Note that here $(x_n)$ may not be Bourbaki quasi-Cauchy in $U$.)
\par
$(6)$ $X^{'}$ is Bourbaki quasi-precompact in $X$ and every sequence $(x_n)$ of isolated points without any Bourbaki quasi-Cauchy subsequence in $X$ satisfies the property that $\inf\{I(x_n):n\in \mathbb{N}\}>0$.\par
$(7)$ Every discrete Bourbaki quasi-complete subset $A$ of $X$%for which every Bourbaki quasi-Cauchy sequence of $X$ contained in $A$ has a cluster point in $A,$ %(it means that every Bourbaki quasi-Cauchy sequence $(x_n)$ in $X$ with $x_n\in A\ \forall\ n\in \mathbb{N}$, has a cluster point in $A$. Note that here $(x_n)$ may not be Bourbaki quasi-Cauchy in $A$.)
 is uniformly discrete.\par
$(8)$ Every Bourbaki quasi-complete subset $A$ of $X$ is an $UC$ space.% provided every Bourbaki quasi-Cauchy sequence of $X$ contained in $A$ has a cluster point in $A.$ %(it means that every Bourbaki quasi-Cauchy sequence $(x_n)$ in $X$ with $x_n\in A\ \forall\ n\in \mathbb{N}$, has a cluster point in $A$. Note that here $(x_n)$ may not be Bourbaki quasi-Cauchy in $A$.)

\end{Theorem}
\begin{proof}
$(1)\Rightarrow(2)$ On the contrary suppose that there is a pseudo-Cauchy sequence of distinct terms in $(X,d)$ which has no Bourbaki quasi-Cauchy subsequence in $X.$ If needed, by passing through a suitable subsequence we may assume that $d(x_{2n-1},x_{2n})<\frac{1}{n}$ for all $n\in \mathbb{N}.$ As $(x_{2n})$ has no Bourbaki quasi-Cauchy subsequence in $X$, it is not Bourbaki quasi-precompact in $X.$ Hence there exists an $\varepsilon >0$ and passing through a subsequence we have $x_{2i}\notin B^{\infty}_{\varepsilon}(x_{2j})$ whenever $i\neq j.$ Choose a sequence of positive real numbers $(\delta_n)$ where $\delta_n=\frac{1}{2}\min\{\varepsilon,d(x_{2n},x_{2n-1})\}.$ Define $f:X\to \mathbb{R}$ by $f(x)=\frac{\delta_n-d(x_{2n},x)}{\delta_n}$ if there exists $n\in \mathbb{N}$ such that $x\in B_{\delta_n}(x_{2n})$ and $f(x)=0,$ otherwise. Note that any quasi-Cauchy sequence in $X$ can share an infinite subsequence with at most one open ball from the family $\{B_{\delta_n}(x_{2n}):n\in \mathbb{N}\}$ and at the same time can only meet finitely many balls of this family. Because otherwise there would exist $x_{2i}$ and $x_{2j}$ $(i\neq j)$ which can be joined by an $\varepsilon$-chain, which contradicts our construction. So for each quasi-Cauchy sequence $(z_n)$ in $X$ there exists $i\in \mathbb{N}$ such that $|f(z_{n+1})-f(z_n)|\leq \frac{1}{\delta_i}d(z_{n+1},z_n)$ for all but finitely many $n$. Therefore $f$ is ward continuous but observe that $f$ cannot be uniformly continuous as $f(x_{2n})=1$ and $f(x_{2n-1})=0$ for all $n\in \mathbb{N}$ which contradicts (1).\par
$(2)\Rightarrow(3)$ Let $(x_n)$ be a sequence for which $I(x_n)\to 0.$ If $(x_n)$ has a constant subsequence then we are done. So we can assume that $x_n$'s are distinct and passing to a subsequence we have $I(x_n)<\frac{1}{n}$ for all $n\in \mathbb{N}.$ Choose $y_n\in X\setminus \{x_n\}$ with $d(x_n,y_n)<\frac{1}{n}.$ Define a new sequence $(z_n)$ by taking $z_{2n}=x_n$ and $z_{2n-1}=y_n$ for all $n\in \mathbb{N}.$ Clearly $(z_n)$ is a pseudo Cauchy sequence of distinct terms which consequently implies that $(x_n)$ has a Bourbaki quasi-Cauchy subsequence in $X.$\par
$(3)\Rightarrow (4)$ Without any loss of generality take a sequence $(x_n)$ of distinct points in $X^{'}$. Clearly $I(x_n)\to 0.$ From the given condition $(x_n)$ has a Bourbaki quasi-Cauchy subsequence in $X$ and so $X^{'}$ is Bourbaki quasi-precompact in $X.$\par
For the next part, suppose that $U\subseteq X$ is an infinite Bourbaki quasi-complete subset of $X$ with $X^{'}\cap U=\phi$. % and every Bourbaki quasi-Cauchy sequence of $X$ contained in $U$ has a cluster point in $U.$%
 If $U$ is not uniformly discrete, then for each $n\in \mathbb{N}$ there exist $x_n,y_n\in U$ such that $d(x_n,y_n)<\frac{1}{n}.$ Clearly $I(x_n)\to 0.$ So $(x_n)$ has a Bourbaki quasi-Cauchy subsequence in $X$ and from the given condition $(x_n)$ has a cluster point in $U,$ which contradicts that $X^{'}\cap U=\phi.$ Hence $U$ is uniformly discrete.\par
$(4)\Rightarrow(5)$ The first part is obvious. For the second part, on the contrary,  suppose that $inf\{I(x):x\in U\}=0.$ Choose $x_n\in U$ with $I(x_n)<\frac{1}{n},$ and correspondingly a sequence $(y_n)$ such that $y_n\in X\setminus \{x_n\}$ with $d(x_n,y_n)<\frac{1}{n}.$ Let $S=\{y_n:n\in \mathbb{N}\}.$ Evidently $S$ must be infinite, whereas $S\cap X^{'}$ is finite, because if $S\cap X^{'}$ is infinite then $X^{'}$ contains a subsequence of $(y_n).$ As $X^{'}$ is Bourbaki quasi-precompact in $X$ so passing to a subsequence we can conclude that $(x_n)$ is a bourbaki quasi-Cauchy in $X.$ Moreover from the given conditions $(x_n)$ has a cluster point in $U,$ which again contradicts that $U\cap X^{'}=\phi.$ Hence there exists $n_0 \in\mathbb{N}$ such that $y_n\notin X^{'}$ for all $n\geq n_0.$ Note that $(y_n)_{n\geq n_0}$ has no Bourbaki quasi-Cauchy subsequence in $X$ as otherwise $(x_n)$ has a cluster point in $U.$ Therefore $U\cup \{y_n:n\geq n_0\}$ is an infinite Bourbaki quasi-complete subset of $X$, disjoint from $X^{'}$. So it must be uniformly discrete but this is not true as $d(x_n,y_n)\to 0.$ Hence $\inf\{I(x):x\in U\}>0.$\par
$(5)\Rightarrow (6)$  obvious.\par
$(6)\Rightarrow (7)$ Let $A$ be a discrete Bourbaki quasi-complete subset of $X$. $A\cap X^{'}$ must be finite, because otherwise from Bourbaki quasi-precompactness of $X^{'}$ it follows that $A$ has a cluster point in itself, which contradicts that $A$ is discrete. Hence there exists $\delta_1>0$ such that $d(x,A\setminus \{x\})\geq \delta_1$ for all $x\in A\cap X^{'}.$ We claim that $\inf\{I(x):x\in A\setminus X^{'}\}>0.$ If $\inf\{I(x):x\in A\setminus X^{'}\}=0$, then one can find $x_n\in A\setminus X^{'}$ such that $I(x_n)<\frac{1}{n}\ \forall\ n\in\mathbb{N}.$ Moreover $(x_n)$ cannot have a Bourbaki quasi-Cauchy subsequence in $X$, as otherwise by Bourbaki quasi-completeness of $A$, $(x_n)$ has a cluster point in $A$ and this again contradicts that $A$ is discrete. Now by $(6)$ $\inf\{I(x_n):n\in \mathbb{N}\}>0,$ which again contradicts our assumption that $I(x_n)<\frac{1}{n}\ \forall\ n\in\mathbb{N}.$ Hence we can conclude that $\inf\{I(x):x\in A\setminus X^{'}\}>0.$ Let $\inf\{I(x):x\in A\setminus X^{'}\}=\delta_2.$ Taking $\delta=\min\{\delta_1,\delta_2\}$ we observe that $d(x,y)\geq \delta$ whenever $x\neq y$ and $x,y\in A.$ So $A$ is uniformly discrete.\par
$(7)\Rightarrow (8)$ Suppose that $A$ is a Bourbaki quasi-complete subset of $X$ and $f:A\to \mathbb{R}$ is a continuous function. If possible assume that $f$ is not uniformly continuous. Then there exist $\varepsilon >0$ and two sequences $(x_n)$ and $(y_n)$ in $A$ such that $d(x_n,y_n)\to 0$ whereas $|f(x_n)-f(y_n)|\geq \varepsilon\ \forall n\in \mathbb{N}.$ $(x_n)$ and $(y_n)$ have no Bourbaki quasi-Cauchy subsequence in $X$ as otherwise they will have a common cluster point, contradicting the continuity condition of  $f.$ Let $S=\{x_n:n\in \mathbb{N}\}\cup \{y_n:n\in \mathbb{N}\}$. From $(7)$ we can then conclude that $S$ is uniformly discrete but this again contradicts that $d(x_n,y_n)\to 0.$ Hence $f$ is uniformly continuous and this implies that $A$ is an $UC$ space.\par
$(8)\Rightarrow (1)$ Let $f$ be a real valued ward continuous function defined on $X$. On the contrary suppose that $f$ is not uniformly continuous. Then there exist $\varepsilon >0$ and two sequences $(x_n)$ and $(y_n)$ in $X$ such that $d(x_n,y_n)\to 0$ but $|f(x_n)-f(y_n)|\geq \varepsilon\ \forall n\in\mathbb{N}.$ $(x_n)$ and $(y_n)$ have no Bourbaki quasi-Cauchy subsequence in $X$ as otherwise by passing to a subsequence we obtain a quasi-Cauchy sequence $(z_n)$ with the property that $z_{n_k}=x_k$ and $z_{n_k+1}=y_k$ for all $k\in \mathbb{N},$ which contradicts ward continuity of $f.$ Taking $S=\{x_n:n\in \mathbb{N}\}\cup\{y_n:n\in\mathbb{N}\}$, from $(8)$ we can conclude that $S$ is $UC.$ Then clearly restriction of $f$ on $S$ is uniformly continuous which again contradicts our assumption. Hence $f$ must be uniformly continuous.
\end{proof}
Before proceeding to establish our next result, we recall the following definitions.
\begin{itemize}
\item A space $X$ is said to be straight if whenever $X$ is the union of two closed sets, then $f\in C(X)$ is uniformly continuous iff its restriction to each of the closed sets is uniformly continuous \cite{bd1}.
\item Let $(X,d)$ be a metric space. A pair $C^{+},C^{-}$ of closed sets of $X$ is said to be $u$-placed if $d(C^{+}_{\varepsilon}, C^{-}_{\varepsilon}) > 0$ holds for every $\varepsilon>0,$ where $C^{+}_{\varepsilon}=\{x\in C^{+}:d(x,C^{+}\cap C^- ){\geq} \varepsilon\}$ and  $C^{-}_{\varepsilon}=\{x\in C^{-}:d(x,C^{+}\cap C^{-}) \geq \varepsilon\}$ \cite{bd1}.
\item A space $X$ is said to be $W\mbox{-}$straight if whenever $X$ is the union of two closed sets, then $f\in C(X)$ is ward continuous iff its restriction to each of the closed sets is ward continuous \cite{PDSPNA2}.
\item Let $(X,d)$ be a metric space. A pair $C^{+},C^{-}$ of closed sets of $X$ is said to be $qc$-placed if $C^{+}_{\varepsilon}, C^{-}_{\varepsilon}$ do not have any common quasi-Cauchy sequence for every $\varepsilon>0$ \cite{PDSPNA2}.
\end{itemize}
\begin{Theorem}\cite{bd1} A metric space $(X,d)$ is straight iff every pair of closed subsets, which form a cover of $X$, is $u$-placed.
\end{Theorem}
\begin{Theorem}\cite{PDSPNA2}
$(1)$ Suppose that $(X,d)$ is a metric space and every pair of closed subsets $C^{+},C^{-}$ of $X$, which form a cover of $X$, is $qc$-placed. Then $X$ is $W\mbox{-}$straight.\par
$(2)$ Suppose that $X$ is $W\mbox{-}$straight. Then for any closed cover $(C^+,C^-)$ of $X$ and for any $\varepsilon>0,$ $C_{\varepsilon}^+$ and $C_{\varepsilon}^-$ cannot be connected through a quasi-Cauchy sequence.
\end{Theorem}
\begin{Theorem}
Let $(X,d)$ and $(Y,\rho)$ be two metric spaces. Then the following conditions are equivalent.\par
$(1)$ Every ward continuous function $f:(X,d)\to (Y,\rho)$ is uniformly continuous.\par
$(2)$ Every real valued ward continuous function defined on $X$ is uniformly continuous.\par
$(3)$ Every sequence $(x_n)$ in $X$ satisfying $\displaystyle{\lim_{n\to \infty}}I(x_n)=0$ has a Bourbaki quasi-Cauchy subsequence in $X.$\par
$(4)$ Every $W$-straight subspace of $X$ is straight.\par
$(5)$ Any two subsets $A,B\subset X$ with $d(A,B)=0$ are connected through a quasi-Cauchy sequence.
\end{Theorem}
\begin{proof}
$(1)\Rightarrow (2)$ is obvious and $(2)\Rightarrow (3)$ follows from Theorem \ref{d}.\par
$(3)\Rightarrow (4)$ Let $A\subseteq X$ be a $W$-straight space and let $(C^+,C^-)$ be a closed cover of $A.$ If possible, suppose that there exists an $\varepsilon >0$ for which we have $d(C^+_{\varepsilon},C^-_{\varepsilon})=0.$ Then one can find two sequences $(x_n)$ and $(y_n)$ such that $x_n\in C_{\varepsilon}^+$, $y_n\in C_{\varepsilon}^-$ where $d(x_n,y_n)\to 0.$ As $I(x_n)\to 0,$ from the given condition and passing to a subsequence we can assume that $(x_n)$ is Bourbaki quasi-Cauchy in $X$. Consequently there exists a quasi-Cauchy sequence $(z_k)$ such that $z_{k_n}=x_n$ for some subsequence $(z_{k_n})$ of $(z_k).$ Let us define a new sequence $(z^{'}_k)$ by taking $z_{k_n+1}^{'}=y_n$ for all $n\in\mathbb{N}$ and $z^{'}_k=z_k,$ otherwise. Clearly $(z_k^{'})$ is a quasi-Cauchy sequence connecting $C^+_{\varepsilon}$ and $C^-_{\varepsilon}$ which contradicts the $W$-straightness of $A.$ Hence $d(C^+_{\varepsilon},C^-_{\varepsilon})>0$ for all $\varepsilon >0$, which implies that $A$ is straight.\par
$(4)\Rightarrow (5)$ Let $A, B$ be two non empty subsets of $X$ for which $d(A,B)=0.$ On the contrary we assume that $A$ and $B$ cannot be connected through a quasi-Cauchy sequence. Observe that there exist two sequences $(x_n)\subseteq A$ and $(y_n)\subseteq B$ with $d(x_n,y_n)\to 0.$ But $(x_n)$ cannot have any Bourbaki quasi-Cauchy subsequence in $X$, because otherwise passing to a subsequence we would obtain a quasi-Cauchy sequence $(z_k)$ such that $z_{k_n}=x_n$ and $z_{k_n+1}=y_n$ for some subsequence $(z_{k_n})$ of $(z_k).$ This would contradict our assumption that $A,B$ cannot be connected through a quasi-Cauchy sequence. Take $S=\{x_n:n\in \mathbb{N}\}\cup \{y_n:n\in \mathbb{N}\}.$ Now $S$ is $WC$ as every subsequence of a quasi-Cauchy sequence in $S$ is eventually constant and so $W$-straight. But $S$ is not straight as is evident by taking $C^+=\{x_n:n\in \mathbb{N}\}$ and $C^-=\{y_n:n\in \mathbb{N}\},$ and noting that $d(C^+,C^-)=0.$ Hence $A,B$ can be connected through a quasi-Cauchy sequence.\par
$(5)\Rightarrow (1)$ Let $f:X\to Y$ be a ward continuous function. If possible let us  assume that $f$ is not uniformly continuous. Then there exist $\varepsilon>0$ and two sequences $(x_n)$ and $(y_n)$ in $X$ such that $d(x_n,y_n)\to 0$ while  $\rho(f(x_n),f(y_n))\geq \varepsilon$ for all $n\in\mathbb{N}$. Let us take $A=\{x_n:n\in \mathbb{N}\}$ and $B=\{y_n:n\in \mathbb{N}\}.$ By the given condition, $A,B$ can be connected through a quasi-Cauchy sequence. Then we can construct a quasi-Cauchy sequence $(z_n)$ with the property that $z_{n_k}=x_{m_k}$ and $z_{n_k+1}=y_{m_k}$ for some subsequences $(z_{n_k}),(x_{m_k})$ and $(y_{m_k})$ of $(z_n),(x_m)$ and $(y_m)$ respectively. But this contradicts the fact that $f$ is ward continuous. Hence $f$ must be uniformly continuous.
\end{proof}

\section{Quasi-Cauchy Lipschitz functions}
\label{}
In analysis there is a well-known group of continuous functions which is even stronger than uniformly continuous functions, namely Lipschitz functions. In \cite{Beer1,Beer2,Beer3,g1}, Beer, Garrido and Jaramillo considered various kinds of Lipschitz-type functions, the definitions of which are recalled below.
\begin{Definition} Let $(X,d)$ and $(Y,\rho)$ be two metric spaces.
	A function $f:(X,d)\to (Y,\rho)$ is said to be:\par
	$(1)$ Lipschitz if there exists $K>0$ such that $\rho(f(x),f(y))\leq Kd(x,y)\ \forall x,y\in X.$\par
	$(2)$ Lipschitz in the small if there exist $\delta >0$ and $K>0$ such that $\rho(f(x),f(y))\leq Kd(x,y)$ whenever $d(x,y)< \delta.$\par
	$(3)$ Uniformly locally Lipschitz if there exists $\delta>0$ such that for every $x\in X,$ there exists $K_x>0$ with $\rho(f(u),f(w))\leq K_xd(u,w)$ whenever $u,w\in B_{\delta}(x).$\par
	$(4)$ Cauchy-Lipschitz if $f$ is Lipschitz when restricted to the range of each Cauchy sequence $(x_n)$ in $X.$\par
	$(5)$ Locally Lipschitz if for each $x\in X,$ there exists $\delta_x>0$ such that $f$ restricted to $B_{\delta_x}(x)$ is Lipschitz.
	
\end{Definition}
One can immediately note that every Lipschitz in the small function is uniformly locally Lipschitz. Moreover, it is shown in \cite{Beer3} that the collection of all Cauchy-Lipschitz functions is contained in the class of all locally Lipschitz functions which also contains the class of all uniformly locally Lipschitz functions. But the converse implications are not generally true. As our main objective in this article is to ascertain the role of quasi-Cauchy sequences in different spheres, naturally one can ask what would happen if the notion described in $(4)$ above can be modified in terms of quasi-Cauchy sequences and with precisely this in mind, we introduce the following notion.
\begin{Definition}
	A function $f:(X,d)\to (Y,\rho)$ is said to be quasi-Cauchy Lipschitz if for any quasi-Cauchy sequence $(x_n)$ in $X$ there exists $\lambda >0$ such that $\rho(f(x_{k}),f(x_{k+1}))\leq \lambda d(x_{k},x_{k+1})$ for all $k\in \mathbb{N}.$
\end{Definition}
Clearly every quasi-Cauchy Lipschitz function is ward continuous. Our main objective in this section is to study circumstances under which this new type of Lipschitz function coincides with any of the existing notion and secondly to investigate the density position of these functions in the space of ward continuous functions. We start with Lipschitz in the small functions and the following sequential characterization of Lipschitz in the small functions will come handy for our said purpose.
\begin{Lemma}
	A function $f:(X,d)\to (Y,\rho)$ is Lipschitz in the small iff for any two sequences $(x_n)$ and $(y_n)$ with $d(x_n,y_n)\to 0$ there exists $\lambda >0$ such that $\rho(f(x_n),f(y_n))\leq \lambda d(x_n,y_n)$ for all $n\in \mathbb{N}.$
\end{Lemma}
\begin{proof}
	Let $f:(X,d)\to (Y,\rho)$ be a Lipschitz in the small function. Then there exist $\delta>0$ and $\lambda >0$ such that $d(x,y)< \delta \Rightarrow \rho(f(x),f(y))\leq \lambda d(x,y).$ Let $(x_n),(y_n)$ be two sequences with $d(x_n,y_n)\to 0.$ Without any loss of generality we can assume that $x_n\neq y_n$ for all $n\in \mathbb{N}.$ Choose $n_0\in \mathbb{N}$ such that $d(x_n,y_n)< \delta$ for all $n\geq n_0$ and consequently $\rho(f(x_n),f(y_n))\leq \lambda d(x_n,y_n)$ for all $n\geq n_0.$ Take $\lambda_0=\max\{\lambda,\displaystyle{\sup_{n<n_0}}\frac{\rho(f(x_n),f(y_n))}{d(x_n,y_n)}\}$. Clearly $\rho(f(x_n),f(y_n))\leq \lambda_0 d(x_n,y_n)$ for all $n\in \mathbb{N}.$\par
	Conversely, suppose that $f$ is not Lipschitz in the small. This means that for each $n\in \mathbb{N},$ $\{\frac{\rho(f(x),f(y))}{d(x,y)}: 0<d(x,y)<\frac{1}{n}\}$ is not bounded. But this implies the existence of two sequences $(x_n)$ and $(y_n)$ with $d(x_n,y_n)\to 0$ satisfying $\frac{\rho(f(x_n),f(y_n))}{d(x_n,y_n)} > n,$ which contradicts the given assumption.
\end{proof}
\begin{Theorem}
	Let $(X,d)$ and $(Y,\rho)$ be two metric spaces. Then\par
	$(1)$ Each Lipschitz in the small function from $X$ to $Y$ is quasi-Cauchy Lipschitz function.\par
	$(2)$ Each quasi-Cauchy Lipschitz function from $X$ to $Y$ is Cauchy-Lipschitz.
\end{Theorem}
\begin{proof}
	$(1)$ is an immediate consequence of Lemma 4.1. For $(2)$, suppose that $f:(X,d)\to (Y,\rho)$ is a quasi-Cauchy Lipschitz function but not a Cauchy-Lipschitz function. Then there exist a Cauchy sequence $(x_n)$ and two subsequences $(y_{k})$ and $(z_{k})$ of $(x_n)$ such that $\frac{\rho(f(y_{k}),f(z_{k}))}{d(y_{k},z_{k})} > k$ for all $k\in \mathbb{N}.$ Define a new sequence $(w_k)$ by taking $w_{2k}=y_{k}$ and $w_{2k-1}=z_{k}$ for all $k\in \mathbb{N}.$ Clearly $(w_k)$ is quasi-Cauchy and so $f$ is not quasi-Cauchy Lipschitz.
\end{proof}
 The following examples illustrate the place of quasi-Cauchy Lipschitz functions vis a vis the notions mentioned above.
\begin{Example}
	$(1)$ Let $X=\mathbb{N}\cup \{n+\frac{1}{n}:n\in \mathbb{N}\}$ endowed with usual metric of $\mathbb{R}.$ The characteristic function $\chi_{\mathbb{N}}$ of $\mathbb{N}$ on $X$ is quasi-Cauchy Lipschitz but not Lipschitz in the small as $\frac{|\chi_{\mathbb{N}}(n+\frac{1}{n})-\chi_{\mathbb{N}}(n)|}{\frac{1}{n}}=n\to \infty.$ Note that $\chi_{\mathbb{N}}$ is uniformly locally Lipschitz by taking $\delta=\frac{1}{4}$ and $K_{n}=K_{n+\frac{1}{n}}=n$ for every $n\in \mathbb{N}.$\par
	$(2)$  Now let $X=\{\sqrt{n}:n\in \mathbb{N}\}$ with usual metric of $\mathbb{R}.$ Observe that the characteristic function $\chi_{\sqrt{2n}}$ of $\{\sqrt{2n}:n\in \mathbb{N}\}$ on $X$ is Cauchy-Lipschitz but not quasi-Cauchy Lipschitz. Here again $\chi_{\sqrt{2n}}$ is uniformly locally Lipschitz as for each $n\in \mathbb{N},$ $B_1(\sqrt{n})$ contains at most finitely many points.\par
	$(3)$ Consider $X=\{1+\frac{1}{2}+....+\frac{1}{n}:n\in \mathbb{N}\}$ with usual metric of $\mathbb{R}.$ Define $f:X\to \mathbb{R}$ by $f(1+...+\frac{1}{n})=\sqrt{n}.$ Clearly $f$ is ward Continuous But $f$ is not quasi-Cauchy Lipschitz as $\frac{\sqrt{n+1}-\sqrt{n}}{\frac{1}{n+1}}=\frac{n+1}{\sqrt{n+1}+\sqrt{n}},$ which produces an unbounded set of numbers as $n$ ranges over $\mathbb{N}.$\par
	$(4)$ Finally let $X=\{ne_1+\frac{1}{n}e_k:k,n\in \mathbb{N}\}$ equipped with sup norm of $\ell^{\infty}$, where $(e_k)$ is the usual basis of $\ell^{\infty}$. Define $f:X\to \mathbb{R}$ by $f(ne_1+\frac{1}{n}e_k)=n^k.$ Note that $f$ is quasi-Cauchy Lipschitz as every quasi-Cauchy sequence is eventually constant. But $f$ is not uniformly locally Lipschitz as $f$ is unbounded on the sets $\{ne_1+\frac{1}{n}e_k:k\in \mathbb{N}\}$ of diameter $\frac{1}{n}.$
\end{Example}
\begin{Remark}
	From above examples we can conclude that set of all quasi-Cauchy Lipschitz functions and the set of all uniformly locally Lipschitz functions both properly contain the set of all Lipschitz in the small functions and on the other side are themselves contained in the set of all Cauchy-Lipschitz functions. Moreover these two classes neither coincide, nor their intersection coincides with the class of Lipschitz in the small functions. Further there exists a Cauchy-Lipschitz function which is neither quasi-Cauchy Lipschitz nor uniformly locally Lipschitz. Consider $X=\{\sqrt{n}e_1+\frac{1}{n}e_k:k,n\in \mathbb{N}\}$ equipped with sup norm of $\ell^{\infty}.$ Define $f:X\to \mathbb{R}$ by $f(\sqrt{n}e_1+\frac{1}{n}e_k)=n^k.$ It can be easily seen that $f$ is the required example.
\end{Remark}
We have already discussed about certain conditions under which ward continuity coincides with other types of continuity in the last section. Now we will show that quasi-Cauchy Lipschitz function coincides with other types of Lipschitz functions under the same conditions. Before going to the results we recall a lemma from \cite{Beer3}.
\begin{Lemma}\cite{Beer3}
	Let $(X,d)$ and $(Y,\rho)$ be two metric spaces. Then $f:X\to Y$ is locally Lipschitz iff the restriction of $f$ to the range of each convergent sequence in $X$ is Lipschitz.
\end{Lemma}
\begin{Theorem}
	Let $(X,d)$ and $(Y,\rho)$ be two metric spaces. Then the following statements are equivalent.\par
	$(1)$ Every locally Lipschitz function $f:(X,d)\to (Y,\rho)$ is quasi-Cauchy Lipschitz.\par
	$(2)$ Every real valued locally Lipschitz function defined on $X$ is quasi-Cauchy Lipschitz.\par
	$(3)$ $X$ is Bourbaki quasi-complete.
\end{Theorem}
\begin{proof}
    $(1)\Rightarrow (2)$ straightforward.\par
	$(2)\Rightarrow (3)$ On the contrary, suppose that there is a Bourbaki quasi-Cauchy sequence $(x_k)$ which has no cluster point in $X$. By Theorem \ref{a} one can find a quasi-Cauchy sequence $(z_n)$ such that $z_{n_k}=x_k$ for some subsequence $(z_{n_k})$ of $(z_n).$ Evidently $(z_{n_k+1})$ has no cluster point. Now we define a new sequence $(y_k)$ by taking $y_{2k-1}=z_{n_k}$ and $y_{2k}=z_{n_k+1}$ for each $k.$ Clearly $(y_k)$ has no cluster point. This fact allows us to choose a sequence of positive real numbers $(\delta_k),$ where $\delta_k=\frac{1}{4}d(y_k,\{y_n:n\neq k\}).$ From this construction it immediately follows that the distance between two open balls $B_{\delta_i}(y_i)$ and $B_{\delta_j}(y_j)$ is positive whenever $i\neq j.$ As a result, for each $x\in X$ there exists $\delta_x >0$ such that $B_{\delta_x}(x)$ intersects at most one member from the family $\{B_{\delta_{k}}(y_{k}):k\in \mathbb{N}\}.$ Define $f:X\to \mathbb{R}$ by $f(x)=k-\frac{k}{\delta_{k}}d(x,y_{k})$ if there exists some $k$ with $d(x,y_{k})<\delta_{k}$ and $f(x)=0,$ otherwise. Note that for each $x\in X$, $f$ is Lipschitz on $B_{\delta_x}(x)$ which implies that $f$ is locally Lipschitz. But $f$ cannot be quasi-Cauchy Lipschitz as $|f(z_{n_k+1})-f(z_{n_k})|=|2k-2k+1|=1,$ where $(z_n)$ is a quasi-Cauchy sequence.\par
	$(3)\Rightarrow (1)$ If possible suppose that $f:(X,d)\to (Y,\rho)$ is a locally Lipschitz function which is not quasi-Cauchy Lipschitz. Then there exist a quasi-Cauchy sequence $(x_n)$ and a subsequence $(x_{n_k})$ of $(x_n)$ such that $\frac{\rho(f(x_{n_k+1}),f(x_{n_k}))}{d(x_{n_k+1},x_{n_k})}\geq k$ for each $k\in \mathbb{N}.$ Now being a Bourbaki quasi-Cauchy sequence, $(x_{n_k})$ must have a cluster point in $X$. By passing to a subsequence we can assume that $(x_{n_k})$ is convergent and consequently $(x_{n_k+1})$ also converges to the same limit. But note that the restriction of $f$ cannot be Lipschitz to the range of the convergent sequence $(x_{n_k})\cup (x_{n_k+1})$ which contradicts that $f$ is locally Lipschitz.
\end{proof}
\begin{Theorem}
	Let $(X,d)$ and $(Y,\rho)$ be two metric spaces. Then the following conditions are equivalent.\par
	$(1)$ Every Cauchy Lipschitz function $f:(X,d)\to (Y,\rho)$ is quasi-Cauchy Lipschitz.\par
	$(2)$ Every real valued Cauchy Lipschitz function defined on $X$ is quasi-Cauchy Lipschitz.\par
	$(3)$ Every Bourbaki quasi-Cauchy sequence in $X$ has a Cauchy subsequence.
\end{Theorem}
\begin{proof}
	$(1)\Rightarrow (2)$ straightforward.\par
	$(2)\Rightarrow (3)$ On the contrary, suppose that there is a Bourbaki quasi-Cauchy sequence $(x_k)$ which has no Cauchy subsequence. Therefore one can choose a $\delta >0$ and passing to a subsequence, we will have $d(x_i,x_j)\geq \delta$ whenever $i\neq j.$ Note that by Theorem \ref{a} one can find a quasi-Cauchy sequence $(z_n)$ such that $z_{n_k}=x_k$ for some subsequence $(z_{n_k})$ of $(z_n).$ Now define $f:X\to \mathbb{R}$ by $f(x)=k-\frac{4k}{\delta}d(x_k,x)$ if $\exists k\in \mathbb{N}$ such that $x\in B_{\frac{\delta}{4}}(x_k)$, and $f(x)=0,$ otherwise. Clearly $f$ is Cauchy-Lipschitz as any Cauchy sequence can share an infinite subsequence with at most one open ball from the family $\{B_{\frac{\delta}{4}}(x_k):k\in \mathbb{N}\}$ and can only meet finitely many such balls. But $f$ fails to be quasi-Cauchy Lipschitz as $\frac{|f(z_{n_k})-f(z_{n_k+1})|}{d(z_{n_k},z_{n_k+1})}=\frac{4k}{\delta}$ for all but finitely many $k$.\par
	$(3)\Rightarrow (1)$ Let $f:(X,d)\to (Y,\rho)$ be a Cauchy-Lipschitz function. Suppose that $f$ is not quasi-Cauchy Lipschitz. Then there exists a quasi-Cauchy sequence $(x_n)$ such that $\frac{\rho(f(x_{n_k+1}),f(x_{n_k}))}{d(x_{n_k+1},x_{n_k})}\geq k$ for some subsequence $(x_{n_k})$ of $(x_n).$ Now $(x_{n_k})$ cannot have any Cauchy subsequence as $f$ cannot be Lipschitz to the range of the sequence $(x_{n_k})\cup (x_{n_k+1}).$ This contradicts the given condition as $(x_{n_k})$ is Bourbaki quasi-Cauchy and so we can conclude that $f$ must be quasi-Cauchy Lipschitz.
\end{proof}
\begin{Theorem}
	Let $(X,d)$ and $(Y,\rho)$ be two metric spaces. Then the following conditions are equivalent.\par
	$(1)$ Every quasi-Cauchy Lipschitz function $f:X\to Y$ is Lipschitz in the small.\par
	$(2)$ Every real valued quasi-Cauchy Lipschitz function $f$ defined on $X$ is Lipschitz in the small.\par
	$(3)$ Every sequence $(x_n)$ in $X$ satisfying $\displaystyle{\lim_{n\to \infty}}I(x_n)=0$ has a Bourbaki quasi-Cauchy subsequence in $X.$
\end{Theorem}
\begin{proof}
	$(1)\Rightarrow (2)$ is obvious and the proof of $(2)\Rightarrow (3)$ is analogous to the proof of $(1)\Rightarrow (2)$ of Theorem \ref{d}.\par
	$(3)\Rightarrow (1)$ Let $f:X\to Y$ be a quasi-Cauchy Lipschitz function. If possible, assume that $f$ is not Lipschitz in the small. Then there exist two sequences $(x_n)$ and $(y_n)$ in $X$ such that $d(x_n,y_n)\to 0$ and $\frac{\rho(f(x_n),f(y_n))}{d(x_n,y_n)}\geq n$ for all $n\in\mathbb{N}$. Consequently from the given condition we can construct a quasi-Cauchy sequence $(z_k)$ with the property that $z_{k_p}=x_{n_p}$ and $z_{k_p+1}=y_{n_p}$ for some subsequences $(z_{k_p}),(x_{n_p})$ and $(y_{n_p})$ of $(z_k),(x_n)$ and $(y_n)$ respectively. But this contradicts the fact that $f$ is quasi-Cauchy Lipschitz. Hence $f$ must be Lipschitz in the small.
\end{proof}
\begin{Theorem}
	Let $(X,d)$ be a metric space. Then every function $f:X\to \mathbb{R}$ is quasi-Cauchy Lipschitz iff restriction of $f$ on $A$ is Lipschitz in the small for every Bourbaki quasi-precompact subset $A$ of $X$.
\end{Theorem}
In \cite{Beer3, g1} it was shown that the set of all real valued locally Lipschitz (Cauchy Lipschitz and Lipschitz in the small) functions on an arbitrary metric space $(X,d)$ are uniformly dense in the set of all real valued continuous (Cauchy continuous and uniformly continuous respectively) functions. Interestingly this same pattern follows in case of the real valued ward continuous functions also. In our final result, we will establish the density of the set of all real valued quasi-Cauchy Lipschitz functions on $X$ in the set of all real valued ward continuous functions.
\begin{Theorem}
	Let $(X,d)$ be a metric space. Then every real valued ward continuous function defined on $X$ can be uniformly approximated by real valued quasi-Cauchy Lipschitz functions.
\end{Theorem}
\begin{proof}
	Let $f$ be a real valued ward continuous function defined on $X$ and let $\varepsilon >0$ be given. First of all for each $n\in \mathbb{Z}$, let us consider the set $C_n = \{x:(n-1)\varepsilon <f(x)<(n+1)\varepsilon\}$ and let us define a function $g_n(x)=\inf\{1,d(x,X\setminus C_n)\}$. Note that by continuity of $f$, the family $\{C_n\}_{n\in \mathbb{Z}}$ satisfies the property that for each $x\in X$, there exists $\delta_x >0$ such that the ball of radius $\delta_x$ with centre $x$ is contained in some $C_m$. Moreover $C_n\cap C_m=\phi$, whenever $|n-m|>1$. So the function $g(x) = \displaystyle\sum_{n \in \mathbb{Z}} g_n(x)$ is well defined and it satisfies that $\delta_x \leq g(x) \leq 2$ for all $x\in X$.

We define $h:X\to \mathbb{R}$ by $h(x) = \frac{1}{g(x)} (\displaystyle\sum_{n \in \mathbb{Z}}n g_n(x))$. Now we are going to prove that $h$ is quasi-Cauchy Lipschitz. Take  a quasi-Cauchy sequence $(x_n)$ in $X$. We claim that there exists $\delta > 0$ such that $f(B_\delta (x_n)) \subseteq (f(x_n)-\frac{\varepsilon}{4},f(x_n)+\frac{\varepsilon}{4})$ for all $n\in\mathbb{N}$. If not, then for each $k\in \mathbb{N}$ there would exist $y_{n_k}\in B_{\frac{1}{k}}(x_{n_k})$ such that $|f(x_{n_k})-f(y_{n_k})|\geq \frac{\varepsilon}{4}$. Now let us define a new sequence  $(z_n)$ by taking $z_{n_k+1}= y_{n_k}$ for each $k\in \mathbb{N}$ and $z_n = x_n,$ otherwise. Clearly $(z_n)$ is quasi-Cauchy and consequently this contradicts the fact that $f$ is ward continuous. Therefore for each $n\in \mathbb{N}$ we can choose $m \in \mathbb{Z}$ such that $B_{\delta} (x_n)$ is contained in $C_m$ and consequently $\delta\leq g(x_n)\leq 2$ for all $n\in \mathbb{N}$. Moreover there exists $n_0 \in \mathbb{N}$ such that $d(x_{n},x_{n+1})<\delta$ for all $n\geq n_0$. Now we will estimate $|h(x_{n})-h(x_{n+1})|$ for every $n\geq n_0$. Clearly for each $n\geq n_0$ there exists $m\in \mathbb{Z}$ such that $x_{n},x_{n+1}$ both are in $C_m$ and so we can obtain that $$|g(x_{n})-g(x_{n+1})|= |(g_{m-1}+g_m+g_{m+1})(x_{n})-(g_{m-1}+g_m+g_{m+1})(x_{n+1})|$$
$$\leq\displaystyle\sum_{i=m-1}^{m+1}|g_i(x_{n})-g_i(x_{n+1})|\leq 3d(x_{n},x_{n+1}).$$ Now from Theorem 1 of \cite{g1}, we can conclude that $|h(x_{n})-h(x_{n+1})|\leq \frac{10}{\delta^2} d(x_{n},x_{n+1})$. Let $\lambda= \max\{\frac{10}{\delta^2},\sup\{\frac{|h(x_{n})-h(x_{n+1})|}{d(x_{n},x_{n+1})}:n<n_0,x_{n+1}\neq x_n\}\}$ and so $|h(x_{n})-h(x_{n+1})|\leq \lambda d(x_{n},x_{n+1})$ for all $n\in \mathbb{N}$. Finally from Theorem 1 of \cite{g1}, $|\varepsilon h(x)-f(x)|< \varepsilon$ for every $x\in X$. Hence $\varepsilon h$ is the required quasi-Cauchy Lipschitz function.
\end{proof}

\end{document}